\documentclass[3p,times,review]{elsarticle}
\usepackage[nofiglist,notablist]{endfloat} 

\usepackage{ecrc}

\volume{00}

\firstpage{1}

\journalname{Signal Processing}

\runauth{}

\jid{procs}

\jnltitlelogo{Signal Processing}

\CopyrightLine{2011}{Published by Elsevier Ltd.}

\usepackage{amssymb}
\usepackage{amsthm}
\usepackage{amsmath}
\usepackage{amsfonts}
\usepackage{algorithm}
\usepackage{algorithmic}
\usepackage{theoremref}
\newtheorem{thm}{Theorem}[section]
\newtheorem{lem}[thm]{Lemma}

\newtheorem{ass}[thm]{Assumption}
\newtheorem{rmk}[thm]{Remark}
\DeclareMathOperator{\E}{E}

\usepackage{color}

\usepackage[figuresright]{rotating}

\begin{document}

\begin{frontmatter}



\dochead{}

\title{Moment Conditions for Convergence of Particle Filters with Unbounded Importance Weights}


\author[lut]{Isambi S. Mbalawata\corref{cor1}}%
\ead{Isambi.Mbalawata@lut.fi}
\author[aalto]{Simo S\"arkk\"a}
\ead{Simo.Sarkka@aalto.fi}
\cortext[cor1]{Corresponding author}
\address[lut]{Department of
  Computational Engineering, Lappeenranta University of Technology,
  P.O.Box 20, FI-53851 Lappeenranta, Finland}
\address[aalto]{Department of Biomedical Engineering and Computational
  Science, Aalto University, P.O.Box 12200, FI-00076 Aalto, Finland}

\begin{abstract}
In this paper, we derive moment conditions for particle filter importance weights, which ensure that the particle filter estimates of the expectations of bounded Borel functions converge in mean square and $L^4$ sense, and that the empirical measure of the particle filter converges weakly to the true filtering measure. The result extends the previously derived conditions by not requiring the boundedness of the importance weights, but only boundedness of second or fourth order moments. We show that the boundedness of the second order moments of the weights implies the convergence of the estimates bounded functions in the mean square sense, and the $L^4$ convergence as well as the empirical measure convergence are assured by the boundedness of the fourth order moments of the weights. We also present an example class of models and importance distributions where the moment conditions hold, but the boundedness does not. The unboundedness in these models is caused by point-singularities in the weights which still leave the weight moments bounded. We show by using simulated data that the particle filter for this kind of model also performs well in practice.
\end{abstract}


\end{frontmatter}

\section{Introduction}
Particle filters are sequential Monte Carlo based methods for numerically solving Bayesian filtering problems by approximating the filtering distribution using a weighted set of Monte Carlo samples $\{ (\tilde{x}_t^{(i)}, \tilde{w}_t^{(i)})~:~i=1,\ldots,N \}$ (see, e.g., \cite{Doucet+Freitas+Gordon:2001,Sarkka:2013}). They approximate the filtering probability measure as a linear combination of delta measures located at the particles $\tilde{x}_t^{(i)}$ with the weights given by $\tilde{w}_t^{(i)}$.

In probabilistic sense, the Bayesian estimation problem can be expressed as state inference in a state space model of the form
\begin{equation}
\begin{split}
x_0 & \sim f_0(x_0),\\
  x_t &\sim f_t(x_t \mid x_{t-1}), \\
  y_t &\sim  g_t(y_t \mid x_t),
\end{split}
\label{eq:ssmodel}
\end{equation}
where $t=0,1,2,\ldots$, $x_t \in \mathbb{R}^n$ is the state of the system, $y_t \in
\mathbb{R}^m$ is the measurement,
$f_0(x_0)$ is the prior probability distribution of $x_0$ at initial time step $t=0$, $f_t(x_t \mid x_{t-1})$ is the
transition probability density modeling the
dynamics of the system, and $g_t(y_t \mid x_t)$ is the conditional
probability density of measurements modeling the distribution of
measurements. In applications, the densities are usually with respect to the Lebesgue measure or the counting measure, but other reference measures are possible as well.

An important feature of any particle filter algorithm is that it
should converge to the correct distribution as the number of particles $N \to \infty$. This property of particle filters is well studied and there exists a number of convergence results for particle filters (see, e.g., \cite{Moral+Guionnet:1999, Moral+Guionnet:2001, Crisan+Doucet:2000, Crisan+Doucet:2002, Doucet+Freitas+Gordon:2001, Chopin:2004, Moral:2004, Douc+Moulines:2008, Douc+Moulines+Olsson:2009, Hu+Schon+Ljung:2008, Hu+Schon+Ljung:2011, Bain+Crisan:2009, Moral:2013, Whiteley:2013, Mbalawata+Sarkka:2013,Mbalawata+Sarkka:2014} and references therein). However, the effect of importance distribution on the convergence is less studied and it is typical either to assume that the dynamic model is used as the importance distribution, leading to so called bootstrap filter, or that the unnormalized importance weights are point-wise bounded. Although in central limit theorem type analysis of particle filters this point-wise boundedness is not always assumed (see, e.g., \cite{Chopin:2004}), it is a standard assumption in $L^p$-type analysis of particle filters \cite{Crisan+Doucet:2000, Doucet+Freitas+Gordon:2001, Crisan+Doucet:2002}.

In \cite{Mbalawata+Sarkka:2014}, we derived novel moment conditions for importance weights which ensured the $L^4$-convergence of the modified particle filter of \cite{Hu+Schon+Ljung:2008} for the case of unbounded test functions. Unfortunately, the results in \cite{Mbalawata+Sarkka:2014} are not directly applicable to standard particle filters. In this paper, we give the proofs for the mean square convergence, $L^4$-convergence, and the empirical measure convergence for the standard particle filter in the case of potentially unbounded importance weights and bounded test function. This enlarges the class of state space models in which particle filters are ensured to converge. Our proof follows the spirit and many of the ideas of the proofs in \cite{Crisan+Doucet:2000,Crisan+Doucet:2002} although the assumptions and the main results are different.
\section{Particle Filtering}
Particle filters are related to the Bayesian filtering problem,  which refers to the construction of the filtering probability density function $p(x_t\mid y_{1:t})$. The construction of $p(x_t\mid y_{1:t})$ is done recursively by Bayesian filtering equations (see, e.g., \cite{Sarkka:2013}). Let $B(\mathbb{R}^n)$ be the set of 
bounded Borel measurable functions on $\mathbb{R}^n$, $\phi\in B(\mathbb{R}^n)$, $\pi_{t|t-1}$ the measure corresponding to the probability density $p(x_t \mid y_{1:t-1})$, and $\pi_{t|t}$ the measure corresponding to the density $p(x_t \mid y_{1:t})$. Then the Bayesian filtering equations for state space model \eqref{eq:ssmodel} can be written as 
\begin{equation}\label{bayesian}
\begin{split}
 (\pi_{t|t-1},\phi) &= (\pi_{t-1|t-1}, f_t \, \phi), \\
   (\pi_{t|t},\phi) &= \frac{(\pi_{t|t-1},\phi \, g_t)}{(\pi_{t|t-1}, g_t)},
\end{split}
\end{equation}
where
$(\pi, \phi) \triangleq \int \phi \, \mathrm{d} \pi$,
$f \, \phi (x) \triangleq \int f(z\mid x) \, \phi(z) \, \mathrm{d}\nu_f(z)$, and $\nu_f$ is the reference measure used for the density $f$. We assume that the state-space model satisfies the sufficient conditions for the Bayesian filtering equations to have solutions which are regular densities with respect to the chosen reference measure. For existence of solution to \eqref{bayesian}, we have to require that $(\pi_{t|t-1}, g_t)>0$.

Due to intractability of Equations \eqref{bayesian}, for most state space models, we usually need to approximate them. A particle filter for approximating the solutions of \eqref{bayesian} is given in Algorithm~\ref{standard}. In this paper, we provide the mean square, $L^4$, and empirical measure convergence results for general importance distribution $q(x_t	\mid x_{t-1},y_{1:t})$,  regardless of the boundedness of the importance weights.
%
%
%
\begin{algorithm}
\caption{Standard particle filter}
\label{standard}
\begin{itemize}
\item At $t=0$,
 for $i=1,\ldots,N$, sample $x^{(i)}_0\sim\pi_{0|0}(\mathrm{d}x_0)$.

\item At $t\geq 1$,
\begin{itemize}

\item  Sample $\tilde{x}^{(i)}_{t}\sim q(x_t \mid x_{t-1}^{(i)},y_{1:t})$,  for $i=1,\ldots,N$.

\item Calculate the unnormalized weights by
	\begin{equation}
	w_t(\tilde{x}_t^{(i)},x_{t-1}^{(i)}) =
	\frac{g_t(y_t \mid \tilde{x}_t^{(i)}) \, f_t(\tilde{x}_t^{(i)} \mid x_{t-1}^{(i)})}{q(\tilde{x}_t^{(i)}
	\mid x_{t-1}^{(i)},y_{1:t})},
	\end{equation}
 for $i=1,\ldots,N$, and define unnormalized empirical measure $\hat{\pi}^N_{t|t}$ as
 \begin{equation}
   \hat{\pi}^N_{t|t} = \frac{1}{N}
   \sum_{i=1}^N w_t(\tilde{x}_t^{(i)},x_{t-1}^{(i)})
   \, \delta_{\tilde{x}_t^{(i)}},
 \end{equation}
 where $\delta_{x}$ denotes a Dirac delta measure concentrated at $x$.
 
\item Normalize the weights by $\tilde{w}_t^{(i)} = \frac{w_t^{(i)}}{\sum_{i=1}^Nw_t^{(i)}}$, where $w_t^{(i)} = w_t(\tilde{x}_t^{(i)},x_{t-1}^{(i)})$, and define empirical probability measure $\tilde{\pi}^N_{t|t}$ as
 \begin{equation}
   \tilde{\pi}^N_{t|t} =
   \sum_{i=1}^N \tilde{w}_t^{(i)} \, \delta_{\tilde{x}_t^{(i)}}.
 \end{equation}

\item Do resampling to obtain the resampled particles $x_t^{(i)}$, and define empirical probability measure $\pi^N_{t|t}$, which is the approximation to the filtering distribution, as
 \begin{equation}
   \pi^N_{t|t} =
   \frac{1}{N} \sum_{i=1}^N \delta_{x_t^{(i)}}.
 \end{equation}
\item $t\leftarrow t+1$
\end{itemize}
\end{itemize}
\end{algorithm}
%
%
%
%
%

%
%
\section{Convergence of Mean Square Error}\label{sec:mse}
In this section, we derive a novel mean square convergence theorem for particle filters. For the Theorem \ref{conv_theorem}, we impose the following assumptions.
\begin{ass}\label{ass1}
The measurement model $g_t$ is bounded, that is,
there exist a constant $c_g < \infty$ such that $\forall t\in \mathbb{N}$, $\forall x \in \mathbb{R}^n$, and $\forall y \in \mathbb{R}^m$ we have
 $g_t(y\mid x)\leq c_g<\infty$.
\end{ass}

\begin{ass}\label{ass4}
   The resampling procedure satisfies (see, e.g., \cite{Crisan+Doucet:2000,Crisan+Doucet:2002} for the sufficient conditions for this):
\begin{equation}
\E\left[\Big| (\pi_{t|t}^N,\phi) - (\tilde{\pi}^N_{t|t},\phi)\Big|^2\right]
  \leq C_{t} \frac{\|\phi\|^2}{N},
\end{equation}
where $\|\phi\| \triangleq  \sup_{x\in \mathbb{R}^n} |\phi(x)|$ and $C_t<\infty$ is a constant.
\end{ass}
\begin{ass}\label{ass3}
The importance density $q$ is satisfies the following condition. 
Let
\begin{equation}
  w_t(x_t,x_{t-1}) = \frac{g(y_t \mid x_t) \, f(x_t \mid x_{t-1})}
  {q(x_t \mid x_{t-1},y_{1:t})}
\label{eq:uiwf}
\end{equation}
be the unnormalized importance weight function, $\forall t\in \mathbb{N}$ and $x_{t-1}\in\mathbb{R}^n$. Then $\E[w_t^2(x_t,x_{t-1})\mid x_{t-1}] \le C_w < \infty$, with the expectation taken over $q(x_t \mid x_{t-1},y_{1:t})$.
\end{ass}
\begin{thm}
\label{conv_theorem}
Provided that Assumptions \ref{ass1}, \ref{ass4} and \ref{ass3} hold for all $t\geq 0$, then there exist a constant
$c_t<\infty$ such that, for bounded function $\phi \in B(\mathbb{R}^n)$
\begin{equation}\label{theorem}
\E\left[\Big| (\pi_{t|t}^N,\phi) - (\pi_{t|t},\phi)\Big|^2\right] \leq c_{t} \frac{\|\phi\|^2}{N}.
\end{equation}
\end{thm}
Using Assumptions \ref{ass1}, \ref{ass4} and \ref{ass3}, we now aim to prove Theorem \ref{conv_theorem}, for the case where the importance weights are not necessarily (point-wise) bounded. In the following we use the notation
$g \triangleq g_t$, $f \triangleq f_t$ and $w_t \triangleq w_t(x_t,x_{t-1})$. Additionally, $\mathsf{F}_{t-1}$ denotes the $\sigma$-field generated by the particles $\{x_{t-1}^{(i)}\}_{i=1}^N$ and $\tilde{\pi}_{t|t}$  the empirical measure before the resampling step.

\begin{proof}
For each step (initialization, prediction, update and resampling step) of Algorithm \ref{standard}, we compute the bound for mean square error. However,
to cope with general importance distribution as in \cite{Mbalawata+Sarkka:2013,Mbalawata+Sarkka:2014}, we combine the prediction and update steps, hence
the Bayesian filtering equations \eqref{bayesian} can be re-written as
\begin{equation} \label{combined_equation}
\begin{split}
  (\pi_{t|t},\phi) = \frac{(\pi_{t-1|t-1}, f\,\phi \, g) }{(\pi_{t-1|t-1}, f\,g)}=
	\frac{(\hat{\pi}_{t|t}, \phi)}{(\hat{\pi}_{t|t}, 1)},
\end{split}
\end{equation}
where	$\hat{\pi}_{t|t}(\mathrm{d}x_t) = (\int w_t(x_t,x_{t-1}) \, q(x_t\mid x_{t-1},y_{1:t})\, \mathrm{d}\pi_{t-1|t-1})\,\mathrm{d}x_t$. 

At initial step, $t=0$, we have $\E\left[\left| (\pi_{0|0}^N,\phi) - (\pi_{0|0},\phi)\right|^2\right] \leq c_{0} \frac{\|\phi\|^2}{N}$, because the $N$ particles from the prior distribution ($\pi_{0|0}^N$) are assumed to be independent and identically distributed.
We now aim to prove the corresponding result for all $t\geq 1$, by using an induction argument. The result follows by proving Lemma \ref{l2} (for combined prediction-update steps) and Lemma \ref{l3} (for resampling step).
\end{proof}
%
%
%
\begin{lem}\label{l2}
Let us assume that for $\phi \in B(\mathbb{R}^n)$ and Assumptions \ref{ass1}, \ref{ass4} and \ref{ass3} hold, we have
\begin{equation}\label{eq_pre}
\E\left[\Big| (\pi_{t-1|t-1}^N,\phi) - (\pi_{t-1|t-1},\phi)\Big|^2\right] \leq c_{t-1} \frac{\|\phi\|^2}{N}.
\end{equation}
Then
\begin{equation}\label{up_pre}
\E\left[\Big| (\tilde{\pi}_{t|t}^N,\phi) - (\pi_{t|t},\phi)\Big|^2\right] \leq \tilde{c}_{t} \frac{\|\phi\|^2 }{N}.
\end{equation}
\end{lem}
\begin{proof}
Given $\mathsf{F}_{t-1}$, the $\sigma$-field generated by $\{x_{t-1}^{(i)}\}_{i=1}^N$, then
\begin{equation}\label{eq:field}
\E[(\hat{\pi}_{t|t}^N , \phi)\mid \mathsf{F}_{t-1} ]  = (\pi_{t-1|t-1}^N,f \, \phi \, g)
\end{equation}
and, from Assumption \ref{ass3}, we can easily show the boundedness of $\E[(w_t^{(i)})^2\mid \mathsf{F}_{t-1}]$:
\begin{rmk} \label{rmk}
Provided that  $\E[(w_t^{(i)})^2\mid x_{t-1}]$ is bounded, then
$\E[(w_t^{(i)})^2\mid \mathsf{F}_{t-1}]$ is bounded as well.
\end{rmk}
We know that
\begin{align}
(\tilde{\pi}_{t|t}^N , \phi) -(\pi_{t|t} , \phi)
&\leq
\frac{\|\phi\|}{(\hat{\pi}_{t|t}, 1)}
\Big[(\hat{\pi}_{t|t}, 1)-(\hat{\pi}_{t|t}^N, 1)\Big] \nonumber \\ &+
\frac{1}{(\hat{\pi}_{t|t}, 1)}
\Big[(\hat{\pi}_{t|t}^N, \phi)-(\hat{\pi}_{t|t}, \phi)\Big],
\label{eq2}
\end{align}
where $(\hat{\pi}_{t|t}, 1) = (\pi_{t|{t-1}}, g) > 0$ by our assumptions. 
To prove Equation~\eqref{up_pre}, we need to evaluate the bounds for
$\E[|(\hat{\pi}_{t|t}^N, \phi)-(\hat{\pi}_{t|t}, \phi)|^2]$ and
$\E[|(\hat{\pi}_{t|t}^N, 1)-(\hat{\pi}_{t|t}, 1)|^2]$. We first evaluate the bound for the former expression, from which the latter will follow by setting $\phi=1$.
We define $(\hat{\pi}_{t|t}^N, \phi)-(\hat{\pi}_{t|t}, \phi) = \Pi_1 + \Pi_2$, where
\begin{align*}
\Pi_1&=
(\hat{\pi}_{t|t}^N, \phi)-\E[(\hat{\pi}_{t|t}^N, \phi)\mid \mathsf{F}_{t-1}],\\
\Pi_2&=
\E[(\hat{\pi}_{t|t}^N, \phi)\mid \mathsf{F}_{t-1}] - (\hat{\pi}_{t|t}, \phi).
\end{align*}
We compute $\E[|\Pi_1|^2]$ and $\E[|\Pi_2|^2]$ as follows.
Using the boundedness of $\phi$, Equation \eqref{eq:field} and Remark \ref{rmk}, we get
\begin{align}
&\E\Big[\Big|\Pi_1\Big|^2\mid \mathsf{F}_{t-1}\Big]
\le \frac{1}{N}\E\left[\frac{1}{N}\left(\sum_{i=1}^N\phi(x_t^{(i)}) \,  w_t^{(i)}\right)^2 \mid \mathsf{F}_{t-1} \right]\nonumber\\
&\leq \frac{\|\phi\|^2}{N^2}\sum_{i=1}^N
\E\left[|w_t^{(i)}|^2 \mid \mathsf{F}_{t-1}\right]
\leq 
\tilde{c}_{t_1} \, \frac{\|\phi\|^2}{N}. \label{eqpart1}
\end{align}
For the second part, using Equations \eqref{eq_pre} and \eqref{eq:field} as well as $\|f\phi\|\leq\|\phi\|$, we get
\begin{align}
&\E\Big[\Big|\Pi_2\Big|^2\mid \mathsf{F}_{t-1}\Big]\nonumber\\
&=\E\Big[\Big|(\pi^N_{t-1|t-1} , f \phi g)- (\pi_{t-1|t-1} , f \phi g)\Big|^2\mid \mathsf{F}_{t-1}\Big]
\nonumber\\
&\leq c_{t-1}  \frac{\| \phi\|^2 \,  \|g \|^2 }{N}= \tilde{c}_{t_2} \frac{\|\phi\|^2 }{N}.\label{eqpart2}
\end{align}
Using Minkowski inequality, we combine \eqref{eqpart1} and \eqref{eqpart2} to get
\begin{equation}
\E\Big[\Big|  (\hat{\pi}_{t|t}^N, \phi)-(\hat{\pi}_{t|t}, \phi)\Big|^2\Big]
\leq  \hat{c}_t\frac{\|\phi\|^2}{N}, \label{tele}
\end{equation}
which, with $\phi = 1$,  implies
$\E\Big[\Big|  (\hat{\pi}_{t|t}^N, 1)-(\hat{\pi}_{t|t}, 1)\Big|^2\Big]
\leq  \hat{c}_t\frac{1}{N}$.
Using these results and the Minkowski inequality to \eqref{eq2}:
\begin{align*}
&\left(\E\Big[\Big|(\tilde{\pi}_{t|t}^N , \phi) -(\pi_{t|t} , \phi)\Big|^2 \Big]\right)^{1/2}\nonumber\\
%
%
&\leq
\left(\frac{1}{(\hat{\pi}_{t|t}, 1)}
\sqrt{\hat{c}_t} +\frac{1}{(\hat{\pi}_{t|t}, 1)}
\sqrt{\hat{c}_t} \right)\frac{\|\phi\| }{N^{1/2}}
=\sqrt{\tilde{c}_t}\frac{\|\phi\| }{N^{1/2}},
\end{align*}
%
%
which completes the proof of Lemma \ref{l2}.
\end{proof}
%
%
\begin{lem}\label{l3}
Assume that Assumptions \ref{ass1}, \ref{ass4} and \ref{ass3} hold and that
\begin{equation*}
\E\left[\Big| (\tilde{\pi}_{t|t}^N,\phi) - (\pi_{t|t},\phi)\Big|^2\right] \leq \tilde{c}_{t} \frac{\|\phi\|^2 }{N}.
\end{equation*}
Then
\begin{equation}\label{res}
\E\left[\Big| (\pi_{t|t}^N,\phi) - (\pi_{t|t},\phi)\Big|^2\right] \leq c_{t}\frac{\|\phi\|^2}{N}.
\end{equation}
\end{lem}
\begin{proof}

If we define
$(\pi_{t|t}^N,\phi) - (\pi_{t|t},\phi) =
(\pi_{t|t}^N,\phi) - (\tilde{\pi}_{t|t}^N,\phi) +
(\tilde{\pi}_{t|t}^N,\phi) - (\pi_{t|t},\phi)$, 
then, using Minkowski inequality together with Assumption \ref{ass4} and results of Lemma \ref{l2} we have
\begin{align}
&\left(\E\left[\Big|(\pi_{t|t}^N,\phi) - (\pi_{t|t},\phi) \Big|^2\right]\right)^{1/2}\nonumber\\
&\leq\sqrt{C_t}\frac{\|\phi\| }{N^{1/2}} + \sqrt{\tilde{c}_t}\frac{\|\phi\| }{N^{1/2}}=\sqrt{c_t}\frac{\|\phi\| }{N^{1/2}},\nonumber
\end{align}
which implies that
\begin{align*}
\E\left[\Big|(\pi_{t|t}^N,\phi) - (\pi_{t|t},\phi) \Big|^2\right]
&\leq c_t\frac{\|\phi\|^2 }{N}.
\end{align*}

\end{proof}
%
%
%
%
%
%
\section{The $L^4$ and Empirical Measure Convergence}
In this section we generalize the above $L^2$-convergence results to 
$L^4$-convergence and empirical measure convergence.

\subsection{The $L^4$-Convergence}
To guarantee the $L^4$-convergence results, we use Assumption \ref{ass1} together with the following assumptions.
%
%
\begin{ass}\label{measure4}
The resampling procedure satisfies the condition \cite{Crisan+Doucet:2000}:
\begin{equation}
\E\left[\Big| (\pi_{t|t}^N,\phi) - (\tilde{\pi}^N_{t|t},\phi)\Big|^4\right]
  \leq C_{t} \frac{\|\phi\|^4}{N^2}.
\end{equation}
\end{ass}
\begin{ass}\label{measure0}
Let $w_t(x_t,x_{t-1})$ be the unnormalized importance weight function defined in \eqref{eq:uiwf}, $\forall t\in \mathbb{N}$ and $x_{t-1}\in\mathbb{R}^n$, then $\E[w_t^4(x_t,x_{t-1})\mid x_{t-1}] \le C_w < \infty$, with the expectation taken over $q(x_t \mid x_{t-1},y_{1:t})$.
\end{ass}
%
%
\begin{rmk} \label{rmk1}
Provided that  $\E[(w_t^{(i)})^4\mid x_{t-1}]$ is bounded, then
$\E[(w_t^{(i)})^4\mid \mathsf{F}_{t-1}]$ is bounded as well.
\end{rmk}
For the $L^4$-convergence, we need to prove the following theorem.
\begin{thm}
Provided that Assumptions \ref{ass1}, \ref{measure4} and \ref{measure0} hold for all $t\geq 0$, then
for $\phi \in B(\mathbb{R}^n)$ we have 
\begin{equation}
\E\left[\Big| (\pi_{t|t}^N,\phi) - (\pi_{t|t},\phi)\Big|^4\right] \leq c_{t}\frac{\|\phi\|^4}{N^2}.
\end{equation}
\end{thm}
\begin{proof}
Certainly, this is true for $t = 0$ and the cases for $t \ge 1$ result follows from Lemmas \ref{lemma_2} and \ref{lemma_3} below together with an induction argument.
\end{proof}
%
%
%
\begin{lem}\label{lemma_2}
Assume  Assumptions \ref{ass1}, \ref{measure4} and \ref{measure0} hold and we have
\begin{equation}\label{indu}
\E\left[\Big| (\pi_{t-1|t-1}^N,\phi) - (\pi_{t-1|t-1},\phi)\Big|^4\right] \leq c_{t-1}\frac{\|\phi\|^4}{N^2}.
\end{equation}
Then
\begin{equation}
\E\left[\Big| (\tilde{\pi}_{t|t}^N,\phi) - (\pi_{t|t},\phi)\Big|^4\right]
\leq \tilde{c}_{t}\frac{\|\phi\|^4}{N^2}.
\end{equation}
\end{lem}
\begin{proof}
Recall that we have defined $(\tilde{\pi}_{t|t}^N , \phi) -(\pi_{t|t}^N , \phi)$ in Equation \eqref{eq2} and consider\\
$(\hat{\pi}_{t|t}^N, \phi)-(\hat{\pi}_{t|t}, \phi) =
\left[(\hat{\pi}_{t|t}^N, \phi)-\E[(\hat{\pi}_{t|t}^N, \phi)\mid \mathsf{F}_{t-1}]\right]
+\left[\E[(\hat{\pi}_{t|t}^N, \phi)\mid \mathsf{F}_{t-1}] - (\hat{\pi}_{t|t}, \phi)\right]$.
%
%
%
Using Lemmas 7.1 and 7.2 from \cite{Hu+Schon+Ljung:2008} together with Remark \ref{rmk1}, we can easily deduce
\begin{align}
&\E\left[\Big|(\hat{\pi}_{t|t}^N, \phi)-\E[(\hat{\pi}_{t|t}^N, \phi)\mid \mathsf{F}_{t-1}]\Big|^4\mid \mathsf{F}_{t-1}\right]\nonumber\\
&\leq \frac{16 \|\phi\|^4}{N^4}\left(
\sum_{i=1}^N C_w^4 +
 \left(\sum_{i=1}^NC_w^2\right)^2\right)
= \tilde{c}_1 \, \frac{\|\phi\|^4}{N^2}.
\label{as4_1}
\end{align}
Proceeding as in \eqref{eqpart2} we get that
\begin{equation}
\E\left[\Big|\E[(\hat{\pi}_{t|t}^N, \phi)\mid \mathsf{F}_{t-1}] - (\hat{\pi}_{t|t}, \phi)\Big|^4\mid \mathsf{F}_{t-1}\right] \leq \tilde{c}_2 \, \frac{\|\phi\|^4}{N^2}.
\label{as4_2}
\end{equation}
The result follows by combining \eqref{as4_1} and \eqref{as4_2} with Minkowski's inequality.
\end{proof}

\begin{lem}\label{lemma_3}
Assume  Assumptions \ref{ass1}, \ref{measure4} and \ref{measure0} hold and we have
\begin{equation}
\E\left[\Big| (\tilde{\pi}_{t|t}^N,\phi) - (\pi_{t|t},\phi)\Big|^4\right]
\leq \tilde{c}_{t}\frac{\|\phi\|^4}{N^2}.
\end{equation}
Then
\begin{equation}
\E\left[\Big| (\pi_{t|t}^N,\phi) - (\pi_{t|t},\phi)\Big|^4\right]
\leq c_{t}\frac{\|\phi\|^4}{N^2}.
\end{equation}
\end{lem}
\begin{proof}
Proceeding as the proof of Lemma \ref{l3} but using Assumption \ref{measure4} and Lemma \ref{lemma_2}, the results follows.
\end{proof}
%
%
\subsection{Empirical Measure Convergence}
In this section, we use the $L^4$-results to deduce the empirical measure convergence given in the following theorem.
\begin{thm}
Provided that Assumptions \ref{ass1}, \ref{measure4} and \ref{measure0} hold for all $t\geq 0$, then we have, almost surely,
\begin{equation}
\lim_{N\rightarrow\infty}\pi_{t|t}^N=\pi_{t|t}.
\end{equation}
\end{thm}
\begin{proof}
Using the $L^4$-convergence results, then
the result follows by using the Markov inequality and Borel-Cantelli argument \cite{Crisan+Doucet:2000}.
\end{proof}

\section{Analytical and Numerical Example}
Assume that we have a Cox process, where the a priori dynamics of the state can be modeled as a reflected Brownian motion $x(\tau) \triangleq \eta^{1/2} \, |W(\tau)|$, where $W(\tau)$ is a standard Brownian motion, and the measurements are Poisson distributed with an intensity parameter $\lambda(\tau) = c \, x(\tau)$, where $c > 0$ is a constant, and the measurements are obtained at discrete times $t \in \{1,2,3,4,\ldots\}$. The model can now be formulated as a discrete-time model for the measurement times:
\begin{equation}
\begin{split}
  f(x_t \mid x_{t-1}) &= \frac{1}{\sqrt{2 \, \pi \, \eta}} \,
  \left[
    e^{-\frac{(x_t - x_{t-1})^2}{2\eta}}
  + e^{-\frac{(x_t + x_{t-1})^2}{2\eta}}
  \right], \\
   g(y_t \mid x_t)  &= \begin{cases}
    \lim_{x_t \to 0^+} g(y_t \mid x_t), & \text{if $x_t = 0$},\\
    \frac{(c \, x_t)^{y_t} \exp(-c \, x_t)}{y_t!}, & \text{otherwise,}
 \end{cases}
\end{split}
\nonumber
\end{equation}
where $f(x_t \mid x_{t-1})$ is a density with respect to the Lebesgue measure and $g(y_t \mid x_t)$ with respect to the counting measure. Above, we require that $x_t \ge 0$ for all $t$. The purpose of including $x_t = 0$ as the special case in $g$ is to ensure that it is continuous and bounded in the domain $x_t \ge 0$.

Let us now select a Gamma distribution with constant parameters $\alpha,\beta > 0$ as the importance distribution for a particle filter. Thus the importance sampling density (w.r.t. the Lebesgue measure) is
\begin{equation}
\begin{split}
  q(x_t) &= \frac{\beta^\alpha}{\Gamma(\alpha)}
  \, x_t^{\alpha-1} \, \exp(-\beta \, x_t).
\end{split}
\end{equation}
At some point of time we eventually reach a zero measurement $y_t = 0$. In this case we have for $x_t > 0$
\begin{equation}
  w(x_t,x_{t-1})
  = \frac{1}{\sqrt{2 \, \pi \, \eta}} \, \frac{\exp(-c \, x_t) \, \left[
    e^{-\frac{(x_t - x_{t-1})^2}{2\eta}}
  + e^{-\frac{(x_t + x_{t-1})^2}{2\eta}}
  \right]}{\frac{\beta^\alpha}{\Gamma(\alpha)}
  \, x_t^{\alpha-1} \, \exp(-\beta \, x_t)}.
\end{equation}
Let us assume that $\alpha > 1$. It is now easy to show that for any (finite) selection of $x_{t-1}$ we have
\begin{equation}
\begin{split}
  \lim_{x_t \to 0^+} w(x_t,x_{t-1}) = \infty.
\end{split}
\end{equation}
This happens, because $q(0) = 0$, but the numerator is nonzero. Thus according to the classical result for particle filters \cite{Crisan+Doucet:2000,Doucet+Freitas+Gordon:2001,Crisan+Doucet:2002} the particle filter is not guaranteed to converge in mean square, $L^4$, or empirical measure sense. 

By using $f \le 1 / \sqrt{2 \, \pi \, \eta}$ and combining terms gives
\begin{equation}
\begin{split}
  w(x_t,x_{t-1})
  &\le \frac{\Gamma(\alpha)}{\sqrt{2 \, \pi \, \eta} \, \beta^{-\alpha}}
  \, \exp\left((\beta-c) \, x_t \right)
  \, x_t^{-\alpha+1}.
\end{split}
\end{equation}
Thus we have
\begin{equation}
\begin{split}
  &\E[w^p(x_t,x_{t-1}) \mid x_{t-1}] \\
  &= 
  \int_0^\infty w^2(x_t,x_{t-1}) \, q(x_t) \, \mathrm{d}x_t \\
  &\le 
  \frac{\beta^\alpha}{\Gamma(\alpha)} \,  
  \left( \frac{\Gamma(\alpha)}{\sqrt{2 \, \pi \, \eta} \, \beta^{-\alpha}} \right)^p \\
  &\quad \times \int_0^\infty \exp\left([(p - 1) \, \beta - p \, c] \, x_t \right)
  \, x_t^{(1 - p) \, \alpha + p - 1}.
\end{split}
\end{equation}
%
%
%
Provided that $(p - 1) \, \beta - p \, c < 0$, the above expression is just a constant times the gamma function value $\Gamma((1 - p) \, \alpha + p)$. Recalling that the gamma function is finite for negative arguments other than integers, we can now deduce that even when $y_t = 0$, we have for $p = 2,4$:
\begin{equation}
\begin{split}
  \int_0^\infty (w(x_t,x_{t-1}))^p \, q(x_t) \, \mathrm{d}x_t &\leq c_w < \infty \\
\end{split}
\end{equation}
provided that $\beta < c \, p / (p-1)$, $\alpha > 1$, and $(1 - p) \, \alpha + p$ is not a negative integer. Thus, according to the present theory, the particle filter converges in mean square and $L^4$ sense for bounded Borel functions, and its empirical measure converges.
\begin{figure}[!t]
\centering
\includegraphics[width=0.5\columnwidth]{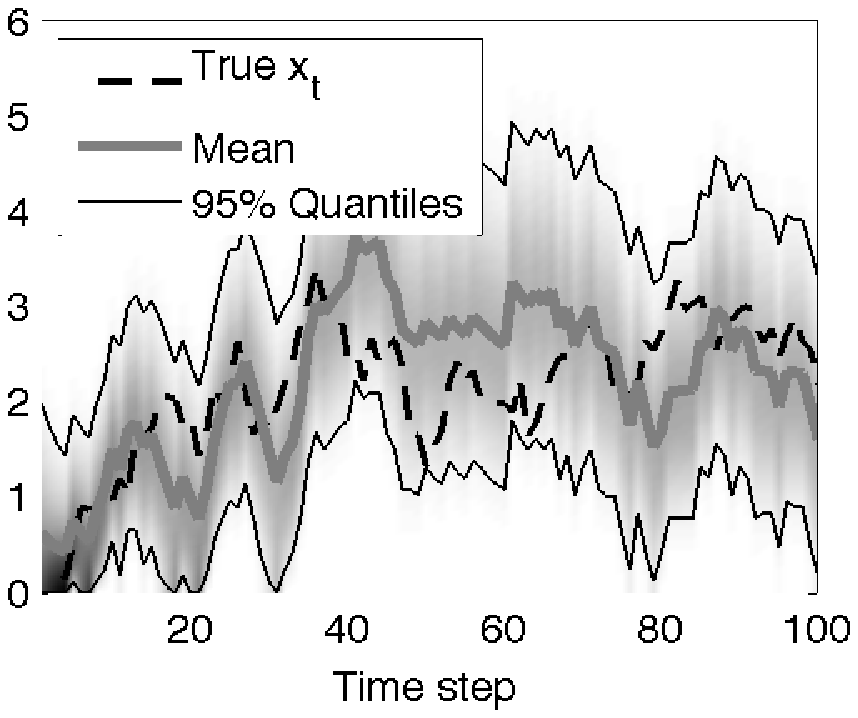}
\includegraphics[width=0.45\columnwidth]{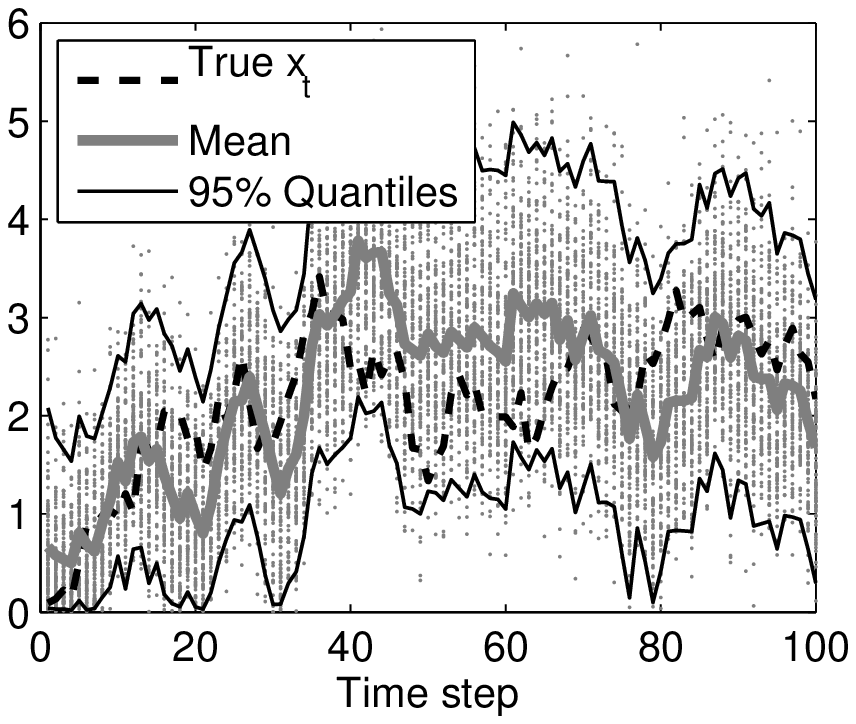}
\caption{{\em Left:} Grid based state estimate of the Cox process. {\em Right:} Particle filter estimate.}
\label{fig:cox}
\end{figure}

\begin{figure}[!t]
\centering
\includegraphics[width=0.45\columnwidth]{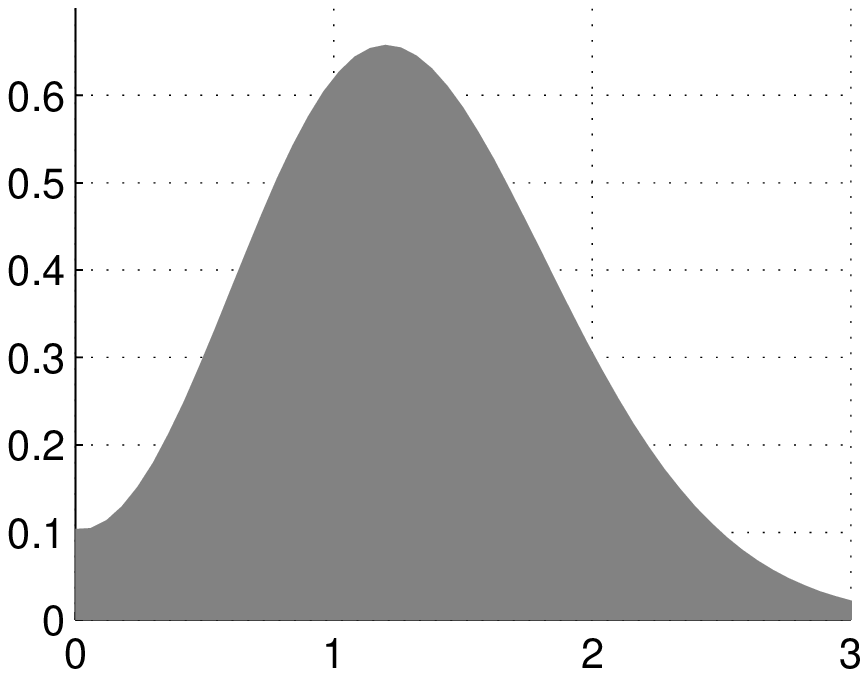}
\includegraphics[width=0.45\columnwidth]{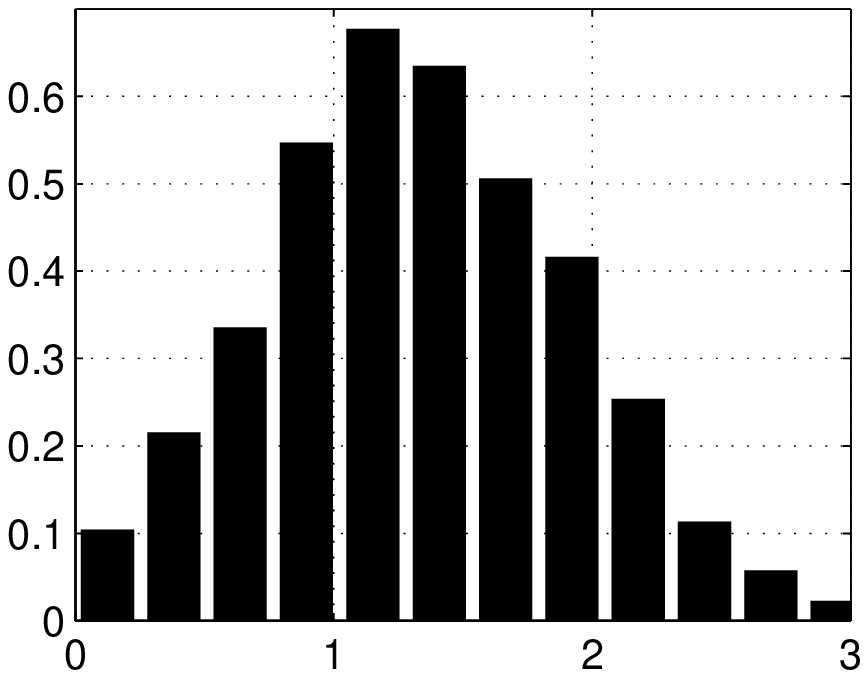}
\caption{{\em Left:} Grid based filter distribution at $t=11$. {\em Right:} Particle filter histogram for the same step.}
\label{fig:cox_hist}
\end{figure}

Because this model is single-dimensional, we can use numerical integration (naive Riemann sum in this case) to approximate the filtering solution in a dense grid. The result of applying the grid filter to a simulated process with $c=1/2$, $q = 1/10$, $x_0 = |\xi|$, where $\xi$ is unit Gaussian, is shown in Figure~\ref{fig:cox} on the left. The right hand side of Figure~\ref{fig:cox} shows the result of a particle filter with $10000$ particles and with the importance distribution parameters $\alpha = 1.5$ and $\beta = 0.5$. For visualization the number of particles is reduced to $100$ per time step. As can be seen, the result is well in line with the grid based result. Figure~\ref{fig:cox_hist} shows the filtering distribution approximations at step $t = 11$, where $y_{11} = 0$ and hence the importance weight is unbounded. The particle filter result is well in line with the grid based result despite the unboundedness of the weight.
%
%
%
\section{Conclusion and Discussion}
We have derived moment conditions for importance weights of particle filters, which ensure that the particle filter estimates of the expectations of bounded Borel functions converge in mean square and $L^4$ sense, and that the empirical measure of the particle filter converges weakly to the true filtering measure. The novel result is that the importance weights do not need to be point-wise bounded. We have also provided an example of a model and a particle filter for which the present theory guarantees the particle filter convergence although the previously developed particle filter theory does not. 

The numerical example showed an example situation when the weight moments can be bounded when the weights are not point-wise bounded. Similar phenomenon is possible whenever there are point-singularities in the weights caused by nulls in the importance distribution. An advantage of the moment conditions is that when the importance distribution is constructed indirectly (as in, e.g., \cite{Sarkka+Sottinen:2008}), the weight moment condition can be easier to check than the point-wise boundedness.

\bibliographystyle{elsarticle-num}
\bibliography{SPL_article}







\end{document}